\documentclass[12pt]{amsart}

\usepackage{amsthm,amsmath,amssymb,amsrefs}

\usepackage{hyperref}
\usepackage{enumitem}
\usepackage[top=1in, bottom=1in, left=1in, right=1in]{geometry}

\usepackage{tikz}
\usetikzlibrary{automata,positioning}

\counterwithin{figure}{section}

\newtheorem{thm}[figure]{Theorem}
\newtheorem{lem}[figure]{Lemma}
\newtheorem{prp}[figure]{Proposition}
\newtheorem{prb}[figure]{Problem}

\theoremstyle{definition}
\newtheorem{dfn}[figure]{Definition}

\theoremstyle{remark}
\newtheorem{rmk}[figure]{Remark}

\NewDocumentCommand{\Aut}{}{\operatorname{Aut}}
\NewDocumentCommand{\Wh}{}{\operatorname{Wh}}

\NewDocumentCommand{\CCLR}{}{strongly reduced}
\NewDocumentCommand{\RCR}{}{weakly reduced}
\NewDocumentCommand{\Rone}{}{partially reduced}

\NewDocumentCommand{\defnsub}{}{\textit}

\newlist{steps}{enumerate}{3}
\setlist[steps]{label=Step \arabic*.,ref=\arabic*,leftmargin=15mm}

\begin{document}

\title{Orbit-blocking words in free groups}

\author[]{Lucy Koch-Hyde}
\address{Department of Mathematics, CUNY Graduate Center, New York,
NY 10016}  \email{lhyde@gradcenter.cuny.edu}

\author[]{Siobh\'an O'Connor}
\address{Department of Mathematics, CUNY Graduate Center, New York,
NY 10016}  \email{doconnor@gradcenter.cuny.edu}

\author[]{
  \'Eamonn Olive
}
\email{ejolive97@gmail.com}

\author[]{Vladimir Shpilrain}
\address{Department of Mathematics, The City College of New York, New York,
NY 10031} \email{shpilrain@yahoo.com}

\begin{abstract}
By strengthening known results about primitivity-blocking words in free groups, we prove that for any element $w$ of a free group of finite rank, there are words that cannot be subwords of any cyclically reduced automorphic image of $w$.
This has implications for the average-case complexity of a variant of Whitehead's problem.

\end{abstract}

\maketitle

\section{Introduction}

Let $F_r$ be a free group with a free generating set (i.e.\ {\it basis}) $x_1, \ldots, x_r$.
Call an element $u \in F_r$ {\it primitive} if there is an automorphism of $F_r$ that takes $x_1$ to $u$.

For a given element $w \in F_r$, we say a word $v$ is $w$-orbit-blocking (or just orbit-blocking if $w$ is clear from the context) if for any automorphism $\varphi \in \Aut(F_r)$, the word $v$ is not a subword of the cyclic reduction of $\varphi(w)$.

{\it Primitivity-blocking} words are the same as $x_1$-orbit-blocking.
The existence of primitivity-blocking words easily follows from Whitehead's observation that the Whitehead graph of any cyclically reduced primitive element is either disconnected or has a cut vertex.
More about this in Section \ref{Back}.

We first show that primitivity-blocking words cannot appear as subwords of some wider classes of primitive elements that are not necessarily cyclically reduced; see Lemma \ref{lem:big boy}, Lemma \ref{thm:RCR no pbsw}, and Theorem \ref{Rone no pbsw}.

We give new examples of primitivity-blocking words and an algorithm for detecting primitivity-blocking words in $F_2$ (Theorem \ref{F2alg}).
Algorithmically detecting primitivity-blocking words in $F_r$ where $r>2$ remains an open problem.

Based on Lemma \ref{lem:big boy}, we show in Section \ref{block2} that for any $w \in F_r$, there are $w$-orbit-blocking words (in fact, depending only on the length of $w$), thus giving a complete solution of Problem (F40) from \cite{Problems}.
Previously \cite{rank2}, we gave a solution in the special case $r=2$ based on a very special and explicit description of bases in $F_2$ from \cite{Cohen}.

In Section \ref{complexity}, we apply our main result to show that the average-case time complexity of the following variant of Whitehead's problem is constant with respect to $n$: given a fixed $u \in F_r$, decide, on an input $v \in F_r$ of length $n$, whether or not $v$ is an automorphic image of $u$.

\section{Background}\label{Back}

The {\it Whitehead graph} $\Wh(w)$ of $w \in F_r$ has $2r$  vertices that correspond to the generators and their inverses.
For each occurrence of a subword $x_i x_j$ in the word $w \in F_r$, there is an edge in $\Wh(w)$ that connects the vertex $x_i$ to the vertex $x_j^{-1}$; if $w$ has a subword $x_i x_j^{-1}$, then there is an edge connecting $x_i$ to $x_j$, etc.
There is one more edge (the external edge): this is the edge that connects the vertex corresponding to the last letter of $w$ to the vertex corresponding to the inverse of the first letter.

It was observed by Whitehead himself in his \textit{cut vertex lemma} (see \cite{Wh2}) that the Whitehead graph of any cyclically reduced primitive element $w$ is either disconnected or has a cut vertex, i.e.\ a vertex that, having been removed from the graph together with all incident edges, increases the number of connected components of the graph.
A short and elementary proof of this result was recently given in \cite{Heusener}, and a more general case of primitive elements in subgroups of $F_r$ was recently treated in \cite{Ascari}.

Thus, for example, if the Whitehead graph of $w$ (without the external edge) is complete (i.e., any two vertices are connected by at least one edge), then $w$ is primitivity-blocking because in this case, if $w$ is a subword of $u$, then the Whitehead graph of $u$, too, is complete and therefore is connected and does not have a cut vertex.
Here are some examples of primitivity-blocking words: $x_1^n x_2^n \cdots x_r^n x_1$ (for any $n \ge 2$), $[x_1, x_2][x_3, x_4]\cdots [x_{n-1}, x_{n}]x_1^{-1}$ (for an even $n$), etc.
Here $[x, y]$ denotes $x^{-1}y^{-1}xy$.

Primitivity blocking was used for the first time in \cite{Bardakov} to show that for any $k$, there is $w \in F_r$ that is not a product of $\le k$ primitive elements.

In \cite{rank2}, we showed that for any $w \in F_2$, there are $w$-orbit-blocking words.
This was based on an explicit description of bases (equivalently, of automorphisms) of $F_2$ from \cite{Cohen}.

\section{Blocking primitivity}\label{block}

\begin{dfn}
Call a basis of $F_r$ \defnsub{\CCLR} if there is no element $g \in F_r$ such that conjugating each element of the basis by $g$ decreases the sum of the lengths of the basis elements.
\end{dfn}

Note that every basis of $F_r$ has a \CCLR\ basis in its conjugacy class.
Since there is a one-to-one correspondence between bases and automorphisms of $F_r$, we will also refer to \CCLR\ automorphisms as those corresponding to \CCLR\ bases.

Then we have the following technical result of independent interest.

\begin{lem}
  \label{lem:big boy}
  Let $B$ be a \CCLR\ basis of $F_r$.
  No primitivity-blocking word appears as a subword of any element in $B$.
\end{lem}

\begin{proof}
  Fix $u$ in $B$.
  We will show it does not contain a primitivity-blocking subword.
  If $u$ is cyclically reduced, then it cannot contain a primitivity-blocking subword by definition.
  Suppose now that $u$ is not cyclically reduced.

  Let $p$ be the first letter appearing in $u$.
  Since $u$ is not cyclically reduced, $p^{-1}$ is the last letter appearing in $u$.
  Now consider two cases:

\medskip

\noindent {\bf 1.} Every element of the basis $B$ either starts with $p$ or ends with $p^{-1}$ (or both).
In this case, conjugating every element of $B$ by $p$ (i.e.\ taking every $z \in B$ to $p^{-1}zp$) will decrease the total length of elements of $B$, contradicting the fact that $B$ is \CCLR.
\medskip

\noindent {\bf 2.} There is an element $v \in B$ that neither starts with $p$ nor ends with $p^{-1}$.
Then consider the element $w=uv$.
Since $v$ has neither left nor right cancellation with $u$, this $w$ is cyclically reduced.
Furthermore, since both $u$ and $v$ belong to the same basis of $F_r$, $w$ must be primitive in $F_r$.
Since $w$ is primitive and cyclically reduced, it, and thus $u$ as well, cannot contain a primitivity-blocking subword.
\end{proof}

Even though Lemma \ref{lem:big boy} suffices to prove our main result (Theorem \ref{blocking} in Section \ref{block2}), we believe the following generalization might be useful, too.

\begin{dfn}
  We call a basis $B$ \defnsub{\RCR} if there is no letter $p$ such that every element of $B$ begins with $p$ and ends with $p^{-1}$.
\end{dfn}

First we need the following

\begin{lem}
  \label{lem:encapsulation}
  Suppose $w,v \in F_r$ such that $w$ and $v$ are not powers of the same  element.
  Let $u(m) = w^mvw^{-m}$.
  Then there exists $m \geq 0$ such that the following are true:
  \begin{enumerate}
    \item The first letters of $u(m)$ and $w$ are the same.
    \item The last letters of $u(m)$ and $w^{-1}$ are the same.
  \end{enumerate}
\end{lem}

\begin{proof}
  Suppose by way of contradiction that for all $m$, $w$ or $w^{-1}$ completely cancels in $u(m+1) = wu(m)w^{-1}$.
  It follows that $|u(m+1)| \leq 2|w|+|u(m)|-2|w| = |u(m)|$.
  Since the sequence $|u(i)|_{i=0}^\infty$ is non-increasing and bounded below, there must be some $k$ such that for all $i \geq k$, $|u(i)| = |u(i+1)|$.
  That is, the sequence must eventually become constant.
  Thus, $\{u(i)\}_{i=0}^\infty$ must eventually form a cycle of equal length words.
  Since conjugation is a one-to-one map, the entire sequence must be a cycle, and so $u(0) = v$ must be a member of the cycle.
  Therefore, each conjugation rotates $v$ by $m$ letters.
  Thus, $w$ and $v$ are powers of the same element, a contradiction with the condition of the lemma.
\end{proof}

\begin{lem}
  \label{thm:RCR no pbsw}
  Let $B$ be a \RCR\ basis of $F_r$, where $r\geq 3$, and let $w \in B$.
Then $w$ does not contain any primitivity-blocking subwords.
\end{lem}

\begin{proof}
  If $w$ is cyclically reduced then it does not contain a primitivity-blocking subword.
  Thus, assume that $w$ begins with $p$ and ends with $p^{-1}$.

  Since $B$ is \RCR, there exists $u \in B$ such that it either does not begin with $p$ or it does not end with $p^{-1}$.
  We will assume, without loss of generality, that $u$ does not begin with $p$.
  Let $v \in B$ be an element such that $v \neq w$ and $v \neq u$.
  Take the word $z = u^mvu^{-m}$ with $m$ given by Lemma \ref{lem:encapsulation}.
  The word $w z$ is a primitive cyclically reduced word containing $w$.
\end{proof}

Note that in $F_2$, a \RCR\ basis can contain a primitivity-blocking word.
For instance, the basis $\{ aba^{-1}, ba^{-1} \}$ is \RCR, but contains $aba^{-1}$, which is primitivity-blocking in $F_2$.

Lemma \ref{thm:RCR no pbsw} can be generalized further, as follows.

\begin{dfn}
  \label{dfn:R1}
  We say that a basis $B$ is \defnsub{\Rone} if there exists a proper subset $X$ of $B$ such that there are non-trivial words $w, v \in \langle X\rangle$ that do not begin with the same letter.
\end{dfn}

\begin{thm}\label{Rone no pbsw}
  Let $B$ be a \Rone\ basis of $F_r$, where $r\geq 3$, and let $u \in B$.
  Then $u$ does not contain a primitivity-blocking subword.
\end{thm}

\begin{proof}
  If $B$ is \RCR, then we are done.
  Thus we will proceed under the assumption that $B$ is not \RCR, and let  $u \in B$ start with a letter $p$ and end with $p^{-1}$.
  We are going to construct a new basis $B' \ni u$ that is \RCR.

  Since $B$ is not \RCR, each basis element of $B$ begins with $p$ and ends with $p^{-1}$.
  We will assume, without loss of generality, that $B \setminus X$ contains just one element, call it $\beta$.
  Let $w \in \langle X\rangle$ be a word beginning with a letter $q \ne p$.
  We will consider two cases:
\medskip

  \noindent
  \textbf{Case 1.} $\beta = u$.\\
  \noindent
Then for each $\gamma \in X$, let $m_\gamma$ be the value given by Lemma \ref{lem:encapsulation} such that $w^{m_\gamma}\gamma w^{-m_\gamma}$ begins with $q$ and ends with $q^{-1}$.
  Let $m$ be the maximum of such numbers $m_\gamma$.
	In this case, we construct the basis $B' = w^mXw^{-m}\cup\{u\}$.
    We can verify this is a basis by conjugating $B$ by $w^m$ to get $w^mBw^{-m}$ and then applying the automorphism which maps $u$ to $w^{-m}u w^{m}$ and maps $X$ to itself.
Then $B'$ is a \RCR\ basis containing $u$.
\medskip

\noindent \textbf{Case 2.} $\beta \neq u$.\\
\noindent  In this case we construct the basis $B' = X\cup\{w^{|\beta|+1}\beta\}$.
Then $B'$ is a \RCR\ basis containing $u$.
\end{proof}

Lemma \ref{lem:big boy} can also be generalized in a different direction.

\begin{dfn}\label{dfn:slender}
  We call a word $w$ \defnsub{slender} if there exists a generator $x_1$ such that the number of letters in $w$ equal to $x_1$ or $x_1^{-1}$ is 1 or 0.
\end{dfn}

\begin{thm}
  \label{thm:CCLR max 1}
  Let $\varphi$ be a \CCLR\ automorphism of $F_r$, where $r \geq 3$.
  Let $w$ be a slender word.
  Then $\varphi(w)$ does not contain a primitivity-blocking subword.
\end{thm}

\begin{proof}
Assume first that some $x=x_i$ appears in $w$ exactly once.
  If $\varphi(w)$ is cyclically reduced, then $\varphi(w)$ is a cyclically reduced primitive word, and thus it cannot contain a primitivity-blocking subword.

Suppose now that $\varphi(w)$ is not cyclically reduced.
  Let $p$ be the first letter of $\varphi(w)$ (and thereby $p^{-1}$ is the last letter of $\varphi(w)$).
  Let $B$ be the basis associated with $\varphi$, i.e.\ $\left\{\varphi(x_j):j=1,2,\dots, r\right\}$.
  If there is some element $\varphi(y)\in B$ not equal to $\varphi(x)$ which neither begins with $p$ nor ends with $p^{-1}$, then $\varphi(w)\varphi(y)$ is a cyclically reduced primitive word.
  Otherwise, since $r \geq 3$ and $\varphi$ is \CCLR, there exists $\varphi(z) \in \left(B\setminus \varphi(x)\right)^{\pm 1}$ that does not begin with $p$.
  Then let $\varphi(s)$ be a basis element not equal to $\varphi(z)$ or $\varphi(x)$, and let $u = \varphi(w)\varphi(z)^m\varphi(s)\varphi(z)^{-m}$, where $m$ is the value given by Lemma \ref{lem:encapsulation}.
  Then $u$ is a cyclically reduced primitive word containing $\varphi(w)$, so $\varphi(w)$ is not primitivity-blocking.

Now suppose that for some $x=x_i$, there is no letter $x$ or $x^{-1}$ in $w$.
Select $y$ from the set of generators and their inverses such that $y\neq x^{\pm 1}$ and there is no cancellation in $\varphi(y)\varphi(w)$.
Since $\varphi$ is \CCLR\ and $r\geq 3$, such a $y$ exists.
Let $m$ be the value given by Lemma \ref{lem:encapsulation} such that $\varphi(y)^{-m}x\varphi(y)^{m}$ has the same last letter as $\varphi(y)$.
Then let $v=y^{-m}xy^mw$.
$v$ contains $x$ exactly once and does not contain its inverse.
Thus by the argument above, $\varphi(v)$ is not primitivity-blocking, so there is a cyclically reduced primitive word containing $\varphi(v)$, and therefore also $\varphi(w)$.
Therefore, $\varphi(w)$ is not primitivity-blocking.
\end{proof}

Just as we will use Lemma \ref{lem:big boy} to produce orbit-blocking words in Theorem \ref{blocking}, we will use Theorem \ref{thm:CCLR max 1} to establish Theorem \ref{thm:blocking2}.

%

\section{On the structure of primitivity-blocking words}
To show that blocking primitivity does not have to be due to connectivity properties of the Whitehead graph, we include the following examples.

\begin{prp}\label{exam} \noindent {\bf (a)}
In the group $F_2$, the words $x_1^k x_2^k$ are primitivity-blocking for any $k \ge 2$.

\noindent {\bf (b)} In the group $F_r$, $r \ge 3$, the words $x_1^k \ldots  x_r^k$ are not primitivity-blocking for any $k \ge 1$.

\end{prp}

\begin{proof}
\renewcommand{\theenumi}{\alph{enumi}}
{\bf (a)} Theorem \ref{basis} says that in a cyclically reduced primitive word in $F_2$, at least one of the generators either only appears with the exponent $1$ throughout or the exponent $-1$ throughout.
Thus, $x_1^k x_2^k$ cannot be a subword of a cyclically reduced primitive element if $k \ge 2$.

\noindent {\bf (b)} Let $r \ge 3$.
The word $u=x_r^k x_2$ is obviously primitive in $F_r$ for any $k \ge 1$.
Applying to $u$ the automorphisms $x_r\mapsto x_i^k x_r x_i^{-k}$, $~x_j \mapsto x_j$ if $j \ne r$, for $i=1,\dots, r-1$ in order gives the element
$$x_1^k\dots x_r^k x_{r-1}^{-k} \dots x_1^{-k} x_2,$$
which is a cyclically reduced primitive word containing $x_1^k\dots x_r^k$ as a subword.
\end{proof}

Part (b) of Proposition \ref{exam} can be generalized as follows:

\begin{prp}\label{XY}
Let $F_r$ where $r \ge 3$ be generated by the set $X\sqcup Y$.
Let $\langle X \rangle$ and $\langle Y \rangle$ denote the subgroups of $F_r$ generated by $X$ and $Y$, respectively.
Let $w=w_Xw_Y$, where $w_X \in \langle X \rangle$, $w_Y \in \langle Y \rangle$.
Then $w$ is not a primitivity-blocking word.

\end{prp}

\begin{proof}
Since $r \ge 3$, we can also assume, without loss of generality, that $|Y| \ge 2$.

Now select $y \in Y$ such that there is no cancellation in $w_Y y$ or in $y w_Y^{-1}$.
Consider the word $v= w_Y^{-1} w_X w_Y y$.
This word is cyclically reduced and contains $w$ as a subword.
To show that $v$ is primitive we apply the automorphism that conjugates all generators in $X$ by $w_Y$ and acts as the identity on $Y$.
Since $Y$ and $X$ are disjoint, this is indeed an automorphism.

Applying this automorphism to $v$ gives $v'=w_Xy$.
Since $v'$ contains the letter $y$ exactly once, it is primitive, and therefore so is $v$.
Thus, $w$ is not primitivity-blocking.
\end{proof}

We note that the shortest  primitivity-blocking word in $F_2$ is $x_1^{-1}x_2x_1$.
Note that the Whitehead graph (with or without the external edge) of this word does have a cut vertex.

It is natural then to look for the shortest  primitivity-blocking words in $F_r$ for $r>2$.

\begin{thm}\label{thm:shortpb}
  The word $w = x_1x_2x_3\dots x_{r-1}x_r^2x_{r-1}\dots x_3x_2x_1^{-1}\in F_r$ is the shortest primitivity-blocking word for $r>2$.
\end{thm}
\begin{proof}
\begin{figure}[b]
\caption{The Whitehead graph of $x_1x_2x_3\dots x_{r-1}x_r^2x_{r-1}\dots x_3x_2x_1^{-1}$}
\begin{tikzpicture}
  \begin{scope}[every node/.style={fill=black,circle,inner sep=1.5}]
    \clip(-0.5,-1)rectangle(2.6,2);
    \foreach \n in {1,...,4}{
      \fill(\n-1,0)node[label={[yshift=-2.5em] $x_\n^{-1}$}](X\n){};
      \fill(\n-1,1)node[label=$x_\n$](x\n){};
    }
    \fill(-1,1)node(xk){};
    \draw(x4)--(X3)--(x2)--(x1)--(X2)--(x3)--(X4);
  \end{scope}
  \draw(3,0.5)node{$\cdots$};
  \begin{scope}[every node/.style={fill=black,circle,inner sep=1.5}]
    \clip(3.4,-1)rectangle(5.5,2);
    \foreach \n in {1,...,3}{
      \fill(\n,0)node[label={[yshift=-2.5em] $x_\n^{-1}$}](X\n){};
      \fill(\n,1)node[label=$x_\n$](x\n){};
    }
    \fill(5,0)node[label={[yshift=-2.5em] $x_r^{-1}$}](Xr){};
    \fill(5,1)node[label=$x_r$](xr){};
    \fill(4,0)node[label={[yshift=-2.75em] $x_{r-1}^{-1}$}](Xr1){};
    \fill(4,1)node[label={[yshift=-0.5em] $x_{r-1}$}](xr1){};
    \fill(3,0)node(Xr2){};
    \fill(3,1)node(xr2){};
    \draw(xr2)--(Xr1)--(xr)--(Xr)--(xr1)--(Xr2);
  \end{scope}
\end{tikzpicture}
\label{fig:Whw}
\end{figure}
  First we will show that there can be no shorter primitivity-blocking words.
  For a word $u$ of length $<2r$ there is an $x_k$ which appears (positively or negatively) one or fewer times in $u$.
  If $x_k$ appears zero times then the word $ux_k$ is primitive, cyclically reduced and contains $u$.
  If $x_k$ appears only once, then $u$ is primitive, and if $u$ is cyclically reduced, then $u$ cannot be primitivity-blocking.
  If $u$ is not cyclically reduced, then let $y$ be a letter such that $y \neq x^{\pm 1}$ and $y$ is not the letter the begins or ends $u$, then $uy$ is cyclically reduced and primitive, and so $u$ is not primitivity-blocking.

  Now we will show that $w$ is primitivity-blocking.
  Assume by way of contradiction that $v$ is a cyclically reduced primitive word which contains $w$ as a subword.
  Assume further that no shorter cyclically reduced primitive word contains $w$ (i.e.\ $v$ is minimal).
  Since $w$ is not cyclically reduced, it must be a proper subword of $v$.
  We remark that the Whitehead graph of $w$ (Figure \ref{fig:Whw}) consists of a single cycle of all the vertices other than $x_1^{-1}$.
  The Whitehead graph of $v$ must contain $\Wh(w)$ as a subgraph.
  Since $w$ begins with $x_1$ and ends with $x_1^{-1}$, there must be an edge in $\Wh(v)$ from $x_1^{-1}$ to some other vertex.
  If there is more than one such edge, then $\Wh(v)$ is connected and has no cut vertices.
  Since we assumed $v$ was a cyclically reduced primitive word, this contradicts the cut vertex lemma.
  So, $\Wh(v)$ has exactly one such edge.
  We will call the other vertex in this edge $y$.
  Applying the automorphism $\varphi =x_1 \mapsto y^{-1}x_1$ then cancels a $y$ or $y^{-1}$ respectively for each $x_1$ or $x_1$ that appears, thus reducing the total length of $v$.
  Furthermore, $\varphi(w)=y^{-1}wy$, of which only the $y^{-1}$ and $y$ cancel in $\varphi(v)$.
  Thus $w$ is a subword of $\varphi(v)$, contradicting our assumption that $v$ was minimal.
\end{proof}

The details of the above proof are heavily inspired by the proof of Theorem 3.7 in \cite{Ascari}.

\begin{rmk}
The word $w = x_1x_2^2x_3^2\dots x_r^2x_1^{-1}$ is also a primitivity-blocking word of the same length as the one in Theorem \ref{thm:shortpb}.
Since it has a Whitehead graph consisting of a cycle of all vertices other than $x_1^{-1}$, begins with $x_1$, and ends with $x_1^{-1}$, the proof of Theorem \ref{thm:shortpb} applies without adjustment to this word as well.
\end{rmk}

\begin{rmk}
A similar word $w = x_1x_2 \dots x_{r-1}x_r^2x_{r-1}\dots x_2x_1$ is not primitivity-blocking in any $F_r$, $r\ge 2$.
  Indeed, for $r=2$, $w$ is a subword of $x_2x_1x_2^2x_1$, which is a cyclically reduced primitive word.

  For $r > 2$, let $u=x_2x_3\dots x_{r-1}x_r^2x_{r-1}\dots x_3x_2$.
  We will now consider the word $v=ux_1ux_1x_3$.
  This  $v$ clearly contains $w$ as a subword and is cyclically reduced.
  Apply the automorphism $\varphi = x_1\mapsto u^{-1}x_1, x_i \mapsto x_i, ~i \ge 2$.
  Then $\varphi(v) = x_1^2x_3$ is a primitive element, and therefore so is $v$.
\end{rmk}

\section{Detecting primitivity-blocking words}\label{algorithmF2}

In light of the results above, it is natural to ask:

\begin{prb}\label{prb:detectpb}
Is there an algorithm that, for a given word $u \in F_r$, would decide whether or not $u$ is primitivity-blocking?
\end{prb}

Below, we give a solution for the case $r=2$.
First we will need the following fact about the structure of bases in $F_2$ that builds on a result of \cite{Cohen}.
For details about this version of the statement, see Lemma 1 of \cite{rank2}.

\begin{thm}[\cite{Cohen}*{p.\ 1}] \label{basis}
    Up to switching $a$ with $b$ or $a$ with $a^{-1}$, or conjugating $u$ and $v$ by the same word, every basis $(u,v)$ of $F(a,b)$ is of the form
    \begin{align*}
        u&=ab^{m_1}\dots ab^{m_p}\\
        v&=\left( b^{r_1}ab^{n_1}\dots a b^{n_q}a b^{r_2} \right)^\varepsilon
    \end{align*}
    where $\varepsilon=\pm 1$ and $\left\{m_1,\dots,m_p\right\} =\left\{r_1 + r_2,n_1,\dots,n_q\right\}=\left\{t,t+1\right\}$ for some fixed integer $t$.
\end{thm}

We will need two more tools that further build on this fact before we can describe our algorithm.

\begin{lem}\label{lem:throwaway}
    Let $w$ be a positive word in $F_2$ whose first syllable is $b^n$.
    Let $q$ be the largest exponent appearing on $b$ in $w$.
    If $n<q$, then $w$ is primitivity-blocking if and only if $b^{-n}w$ is.
    If $n=q$ then $w$ is primitivity-blocking if and only if $aw$ is.
    Analogous statements hold if $b^n$ is the last syllable.
\end{lem}
\begin{proof}
    We will argue assuming that $b^n$ is the first syllable.

    Suppose that $b^{-n}w$ is primitivity-blocking.
    Then, since $b^{-n}w$ is a subword of $w$, $w$ is primitivity-blocking.
    Similarly, if $w$ is primitivity-blocking, then so is $aw$.

    Now suppose that $n<q$ and $w'=b^{-n}w$ is not primitivity-blocking.
    Then there exists $v$ such that $vw'$ is cyclically reduced and primitive with no cancellation between $v$ and $w'$.
    This $w'$ starts with $a$ by our hypothesis.
    Theorem \ref{basis} then implies that the last syllable of $v$ is $b^k$ where $k \geq q-1$.
    Since $n\leq q-1$, $w$ is a subword of $vw'$ and is therefore not primitivity-blocking.

    Finally, suppose that $n=q$ and $w$ is not primitivity-blocking.
    Then there exists $v$ such that $vw$ is cyclically reduced and primitive with no cancellation between $v$ and $w$.
    If the last syllable of $w$ is $b^k$ for some $k<q$, then by the argument above we may consider $\hat{w}=wb^{-k}$ instead.
    Thus assume without loss of generality that $w$ ends with $a$ or $b^q$.
    Suppose somewhere in $w$ $b$ appears to a smaller exponent than $q$.
    Then Theorem \ref{basis} implies that $v$ ends with $a$, so $aw$ is a subword of $vw$ and is therefore not primitivity-blocking.
    Suppose every exponent on $b$ in $w$ is $q$.
    If every exponent on $a$ in $w$ is 1, then either $w$ ends in $b^q$ so $awb$ is primitive, or $w$ ends in $a$, so $awb^{q+1}$ is primitive.
    Otherwise, by Theorem \ref{basis}, every exponent on $b$ in $vw$ must be 1, so $v$ must end in $a$ and we have that $aw$ is not primitivity-blocking.
\end{proof}

\begin{lem}\label{lem:shrink}
    Suppose $u\in F_2$ is of the form $ab^{m_1}\dots ab^{m_p}$, where each $m_i$ is in $\{t,t+1\}$ where $t\geq 0$.
    Denote the automorphism $a\mapsto ab^{-t}$ by $\varphi$.
    Then $u$ is primitivity-blocking if and only if $\varphi(u)$ is.
\end{lem}

\begin{proof}
    Suppose $u$ is not primitivity-blocking.
    We will show that $\varphi(u)$ is not primitivity-blocking either.
    Suppose $uv$ is a cyclically reduced primitive word such that there is no cancellation between $u$ and $v$.
    Then $\varphi(uv)$ is primitive.
    Since $u$ is positive, $v$ must also be positive by Theorem \ref{basis}.
    Thus, $\varphi(v)$ cannot end in $a^{-1}$.
    So, since $\varphi(u)$ starts with $a$, $\varphi(uv)$ is cyclically reduced.
    Furthermore, since $m_p\geq t$, there is no cancellation between $\varphi(u)$ and $\varphi(v)$, so $\varphi(u)$ is a subword of $\varphi(uv)$ and we have that $\varphi(u)$ is not primitivity-blocking.

    We must also show that if $\varphi(u)$ is not primitivity-blocking, neither is $u$.
    Let $\varphi(u)v$ be primitive and cyclically reduced such that there is no cancellation between $u$ and $v$.
    $\varphi^{-1}\left(\varphi(u)v\right) = u\varphi^{-1}(v)$ is then primitive.
    If $\varphi(u)\neq a^p$, then $v$ must be a positive word by Theorem \ref{basis}.
    If $\varphi(u) = a^p$, assume without loss of generality that $v=b$.
    $\varphi^{-1}$ sends $a$ to $ab^t$, so since $a^{-1}$ does not appear in $\varphi(u)v$, we have that $u\varphi^{-1}(v)$ is cyclically reduced.
    Furthermore, $\varphi(u)$ ends in $a$ or $ab$ and the first letter in $v$ must be either $a$ or $b$, so there is no cancellation between $u$ and $\varphi^{-1}(v)$.
    It follows that all of $u$ appears as a subword of $u\varphi^{-1}(v)$, so it cannot be primitivity-blocking.
\end{proof}

\begin{thm}\label{F2alg}
    There is an algorithm that decides, for a given word $u \in F_2$ of length $n$, whether or not $u$ is primitivity-blocking and has the worst-case time complexity $O(n^2)$.
\end{thm}

\begin{proof}
    Call our algorithm $\mathcal{A}$ and let $F_2$ have generators $a$ and $b$.
    After each step we will replace $u$ by the result of that step and still call it $u$.

    \begin{steps}
    \item\label{step:1} $\mathcal{A}$ checks whether there is a generator that either only appears with the exponent $1$ throughout or the exponent $-1$ throughout.
    If not, then Theorem \ref{basis} implies that $u$ is primitivity-blocking, so we stop.
    Otherwise $\mathcal{A}$ swaps $a$ with $b$ and/or swaps $a$ with $a^{-1}$ if necessary so that $a$ only appears with the exponent $1$ in $u$.

    \item\label{step:2} $\mathcal{A}$ checks whether all the exponents on $b$ in $u$ have the same sign.
    If not, then Theorem \ref{basis} implies that $u$ is primitivity-blocking, so we stop.
    Otherwise $\mathcal{A}$ swaps $b$ with $b^{-1}$  if necessary to make all the exponents on $b$ positive.

    \item\label{step:3} $\mathcal{A}$ checks whether there exists an integer $t$ such that every exponent on $b$ in $u$ belongs to $\{t,t+1\}$ except for any times $b$ appears before the first occurrence of $a$ or after the last one.
    (If $b$ always appears to the same exponent, take that exponent to be $t+1$.) If not, then  $u$ is primitivity-blocking by Theorem \ref{basis}.

    \item\label{step:4} $\mathcal{A}$ checks whether there is an exponent on $b$ appearing before the first $a$ or after the last $a$, that is greater than $t+1$.
    If so, then $u$ is primitivity-blocking by Theorem \ref{basis}.

    \item\label{step:5} Suppose $b^k$ is the power of $b$ appearing before the first $a$.
    Then $\mathcal{A}$ checks whether $k$ is the largest exponent on $b$ appearing in $u$.
    If it is, then $\mathcal{A}$ replaces $u$ by $au$.
    Otherwise it replaces $u$ by $b^{-k}u$.
    This does not change the primitivity-blocking property  by Lemma \ref{lem:throwaway}.
    Again using Lemma \ref{lem:throwaway}, if the exponent on $b$ appearing after the last $a$ is less than $t$ (including if it is 0), then $\mathcal{A}$ changes that exponent to $t$.
    Note that at the end of this step, our word is of the form $ab^{m_1}\dots ab^{m_p}$, where each $m_i$ is in $\{t,t+1\}$.

    \item\label{step:6} $\mathcal{A}$ applies the automorphism $a \mapsto ab^{-t}$ to the word $u$ obtained as a result of Step \ref{step:5}.
    By Lemma \ref{lem:shrink}, this does not change whether or not $u$ is primitivity-blocking.

    \item\label{step:7} $\mathcal{A}$ swaps $a$ with $b$ and checks whether the word $u$ has two or fewer syllables.
    If it does, then it is of the form $b^m$ or $b^ma$ for some non-negative $m$.
    These are both subwords of $b^m a$, which is a cyclically reduced primitive word, so neither are primitivity-blocking and thus $\mathcal{A}$ stops and returns that $u$ is not primitivity-blocking.
    If the word $u$ has more than two syllables, then $\mathcal{A}$ goes to Step \ref{step:3}.
    \end{steps}

    Now we note that Step \ref{step:5} of the algorithm $\mathcal{A}$ adds at most $t+1$ to the length of the word and Step \ref{step:6} adds exactly $-pt$ to the length of the word.
    Thus, as long as $p$ is greater than $1$, repeating steps \ref{step:3} through \ref{step:7} decreases the length of the word.
    If $p=0$, the word must be either $a$ or the empty word after Step \ref{step:5}, and either way the algorithm will stop at Step \ref{step:7}.
    If $p=1$, then after Step \ref{step:5}, $u$ is of the form $ab^{m_p}$, so at Step \ref{step:7}, $u$ will have two or fewer syllables, so the algorithm will stop.
    Thus, $\mathcal{A}$ always terminates.
    Furthermore, it follows that each step takes at most linear time in $n$ (the length of the input word), and the number of repetitions of steps \ref{step:3} through \ref{step:7} is also bounded by $n$, so the worst-case time complexity of $\mathcal{A}$ is $O(n^2)$.
\end{proof}

\section{Blocking automorphic orbits}\label{block2}

Now we establish the main result of the paper.

\begin{thm}\label{blocking}
  Let $w\in F_r$ have cyclically reduced length $\ell$.
  Let $\{v_i\}_{i=1}^{\ell+1}$ be a sequence of primitivity-blocking words such that there is no cancellation between adjacent terms of the sequence, nor between $v_{\ell+1}$ and $v_1$.
  Then  $\prod_{i=1}^{\ell+1}v_i$ is $w$-orbit-blocking.
\end{thm}
\begin{proof}
  Let $\varphi$ be an automorphism of $F_r$.
  Let $\bar \varphi$  be a \CCLR\ automorphism of the form $\gamma\circ\varphi$ where $\gamma$ is conjugation by some element of $F_r$.
  Suppose $w=a_1\dots a_\ell$, where each $a_i$ is a letter.
  Then each $\bar \varphi(a_j)$ does not contain any $v_i$ by Lemma \ref{lem:big boy}.
  Thus if some $v_i$ appears in $\bar \varphi(w)$, it cannot be contained entirely within any single $\bar \varphi(a_j)$ and thus must be contained partially in one and partially in an adjacent one.

  Consider $\bar \varphi(w)$ as a cyclic word.
  Since there are $\ell$ pairs of adjacent images of letters under $\bar \varphi$, there are at most $\ell$ non-intersecting words $v_i$ in $\bar \varphi(w)$.

  Note that $\varphi(w)$ and $\bar \varphi(w)$ may differ, but they are in the same conjugacy class, so the cyclic reduction of $\varphi(w)$ also contains at most $\ell$ non-intersecting subwords belonging to $\{v_i\}_{i=1}^{\ell+1}$.
  Therefore, no cyclically reduced element in the orbit of $w$ can contain the product of all $\ell+1$ words $v_i$.
\end{proof}

A sequence of primitivity-blocking words mentioned in the statement of Theorem \ref{blocking} can be built as follows.
Take a particular primitivity-blocking word, e.g.\ $u=x_1^2 x_2^2 \cdots x_r^2 x_1$ (see Section \ref{Back}).
Let $v_i=u$ for all $i=1, \ldots, k+1$.
This sequence of $v_i$ satisfies all conditions of Theorem \ref{blocking}.

\section{Average-case complexity of the Whitehead problem}\label{complexity}

Here we address the computational complexity of the following version of the Whitehead problem on automorphic equivalence in $F_r$: given a fixed $u \in F_r$, decide, on an input $v \in F_r$ of length $n$, whether or not $v$ is an automorphic image of $u$.
We show that the average-case complexity of this version of the Whitehead problem is constant if the input $v$ is a cyclically reduced word.
For a formal definition of the average-case complexity of an algorithm in the context of group theory we refer the reader to \cite{KMSS2}.

This version is a special case of the general Whitehead problem \cite{Wh} that asks, given two elements $u, v \in F_r$,  whether or not $u$ can be taken to $v$ by an automorphism of $F_r$.
The worst-case complexity of the Whitehead problem is unknown in general (cf.\ \cite{Problems}*{Problem (F25)}) but is at most quadratic in $\max(|u|, |v|)$ if $r=2$, see \cite{MS} and \cite{Khan}.

The version of the Whitehead problem that we consider here is different in that the input consists of just one element $v$, and the complexity of an algorithm that solves the problem is a function of the length of $v$, while the length of $u$ is considered constant.

To construct an algorithm with constant average-case complexity, we use the same approach that we used in \cite{rank2} in the special case where the rank of the ambient free group is 2.

Our algorithm will be a combination of two different algorithms running in parallel: one is fast but may be inconclusive, whereas the other one is conclusive but relatively slow.
This idea has been used for group theoretic algorithms since at least \cite{KMSS2}.

Before running the two algorithms, a pre-computation would be performed on $u$.
The algorithm would reduce the length of the fixed word $u$ by the process of ``Whitehead minimization''.
The minimization process successively applies elementary Whitehead automorphisms to the word that reduce its length until its length cannot be reduced any further.
This process takes worst-case quadratic time in $|u|$.
However, since $u$ is a fixed word, this amounts to constant time for the algorithm.
Denote the obtained element of minimum length in the orbit of $u$ by $\bar u$.
Once the word has been minimized, the two algorithms will be run in parallel on the result.

A fast algorithm $\mathcal{T}$ would detect a $\bar u$-orbit-blocking subword $B(\bar u)$ of a (cyclically reduced) input word $v$, as follows.
Let $n$ be the length of $v$.
The algorithm $\mathcal{T}$ would read the initial segments of $v$ of length $k$, $k=1, 2, \ldots,$ adding one letter at a time, and check if this initial segment has $B(\bar u)$ as a subword.
This takes time bounded by $C\cdot k$ for some constant $C$, see e.g.\ \cite{Knuth2}*{p.\ 338}.

The ``usual" Whitehead algorithm, call it $\mathcal{W}$, would minimize $|v|$ taking time quadratic in $|v|$.
Denote the obtained element of minimum length in the orbit of $v$ by $\bar v$.
If $|\bar v| \ne |\bar u|$, then $\mathcal{W}$ stops and reports that $v$ is not in the automorphic orbit of $u$.
If $|\bar v| = |\bar u|$, then the algorithm $\mathcal{W}$ would apply all possible sequences of elementary Whitehead automorphisms that do not change the length of $\bar v$ to see if any of the resulting elements are equal to $\bar u$.
This part may take exponential time in $|\bar v| = |\bar u|$, but since we consider $|u|$ constant with respect to $|v|$ and $|u|$ bounds $|\bar u|$ above, exponential time in $|\bar u|$ is still constant with respect to $|v|$.
Thus, the total time that the algorithm $\mathcal{W}$ takes is quadratic in $|v|$.

To address complexity of the algorithm $\mathcal{T}$, we use a result from
\cite{languages}:

\begin{lem}[\cite{languages}*{Example 1}]\label{csw}
The number of (freely reduced) words of length $L$ with (any number of) ``forbidden"  subwords (e.g.\ $u$-orbit-blocking subwords)  grows exponentially slower than the number of all freely reduced words of length $L$ does.
\end{lem}

In our situation, we have at least one $B(u)$ as a forbidden subword.
Therefore, the probability that the initial segment of length $k$ of the word $v$ does not have $B(u)$ as a subword is $O(s^k)$ for some $s, ~0<s<1$.
Thus, the average time complexity of the algorithm $\mathcal{T}$ is bounded by
$\sum_{k=1}^n  C\cdot k\cdot s^k,$ which is bounded by a constant.

\begin{thm}\label{average-case}
Suppose possible inputs of the above algorithm $\mathcal{A}$ are cyclically reduced words that are selected uniformly at random from the set of cyclically reduced words of length $n$.
Then the average-case time complexity (i.e.\ expected runtime) of the algorithm $\mathcal{A}$, working on a classical Turing machine, is $O(1)$, a constant that does not depend on $n$.
If one uses the ``Deque" (double-ended queue) model of computing \cite{deque} instead of a classical Turing machine, then the ``cyclically reduced" condition on the input can be dropped.
\end{thm}

\begin{proof}
  As stated above, given the existence of orbit-blocking words for all groups $F_r$, where $r \geq 2$ (according to Theorem \ref{blocking}), we can proceed essentially as in the proof of Theorem 3 in \cite{rank2}, with only minimal modifications needed.
  For instance, the first statement quickly follows from Lemma \ref{csw}.
\end{proof}

\subsection{Reducing complexity further}

While the above offers a theoretically fast algorithm, one may be skeptical about its  practical performance.
The algorithm $\mathcal{A}$ has constant time average-case complexity, but it appears that this constant may be quite large. Indeed, our proof of Theorem \ref{blocking} offers 
$w$-orbit-blocking words of length $(2r+1)\cdot |w|$, and thus it is impossible to speed up the whole algorithm $\mathcal{A}$ until the fast part of the algorithm reads a prefix that exceeds this length.
We will offer a few ways to improve practical performance of the fast part of the algorithm.

The first idea is to use shorter primitivity-blocking words to construct our orbit-blocking words.
In Theorem \ref{thm:shortpb} we give the shortest possible primitivity-blocking word, however this word has cancellation when multiplied by itself, and so we cannot simply raise it to a power.
In order to prevent cancellation we can use $v_i=x_1x_2\dots x_r^2\dots x_2x_1^{-1}$ when $i$ is even and  $v_i=x_1^{-1}x_2\dots x_r^2\dots x_2x_1$ when $i$ is odd.

Another option is to decrease the number of primitivity-blocking words required.
By applying Theorem \ref{thm:CCLR max 1} in place of Lemma \ref{lem:big boy}, we have the following theorem:

\begin{thm}\label{thm:blocking2}
  Let $w\in F_r$, $r\geq 3$, be such that its cyclic reduction can be written as a product of $k$ slender words.
  Let $\{v_i\}_{i=1}^{k+1}$ be a sequence of primitivity-blocking words such that there is no cancellation between adjacent terms of the sequence, nor between $v_{k+1}$ and $v_1$.
  Then  $\prod_{i=1}^{k+1}v_i$ is $w$-orbit-blocking.
\end{thm}

The proof follows in the same manner as that of Theorem \ref{blocking}, applying Theorem \ref{thm:CCLR max 1} rather than Lemma \ref{lem:big boy}.

Every word of length $\leq 2r-1$ must necessarily be slender.
Thus partitioning a word into chunks of that size (or smaller) splits it as a product of slender words.
For a word of length $\ell$, Theorem \ref{blocking} implies there is an orbit-blocking word of length $2r\left(\ell+1\right)$ (using the short primitivity-blocking words as described above).
Comparatively, by Theorem \ref{thm:blocking2} a word of length $\ell$ has an orbit-blocking word of length at most $2r\left(\left\lceil\dfrac{\ell}{2r-1}\right\rceil+1\right)$.

Lastly, one could look for multiple orbit-blocking words simultaneously.
Theorems \ref{blocking} and \ref{thm:blocking2} in essence give a limit on the number of disjoint primitivity-blocking words that can appear in the orbit of a given word.
An alternative algorithm could simply count primitivity-blocking words in the input until the number exceeds the bound given by the fixed word.
This highlights the significance of Problem \ref{prb:detectpb}.

\end{document}